\documentclass{amsart}
\usepackage{amsmath,amsthm,amssymb,amscd}
\usepackage{braket,graphicx}
\usepackage{latexsym}
\usepackage{tabls}

\theoremstyle{plain}
\newtheorem{theorem}{Theorem}[section]
\newtheorem{proposition}[theorem]{Proposition}

\newtheorem{lemma}[theorem]{Lemma}

\theoremstyle{definition}
\newtheorem{example}[theorem]{Example}
\newtheorem{definition}[theorem]{Definition}

\def\Im{\operatorname{Im}}
\def\Ker{\operatorname{Ker}}
\def\res{\operatorname{res}}
\def\Sign{\operatorname{Sign}}
\def\Int{\operatorname{Int}}
\def\Diff{\operatorname{Diff}}
\def\half{\operatorname{half}}
\def\id{\operatorname{id}}
\def\gcd{\operatorname{gcd}}
\def\I{\operatorname{I}}
\def\II{\operatorname{II}}
\def\SL{\operatorname{SL}}
\def\ori{\operatorname{ori}}
\def\loc{\operatorname{loc}}
\def\sign{\operatorname{sign}}

\begin{document}
\title{A signature formula for hyperelliptic broken Lefschetz fibrations}

\author{Kenta Hayano}
\address{Department of Mathematics, Graduate School of Science, 
Osaka University, Toyonaka, Osaka 560-0043, Japan}
\email{k-hayano@cr.math.sci.osaka-u.ac.jp}

\author{Masatoshi Sato}
\address{Department of Mathematics, Graduate School of Science,
Osaka University, Toyonaka, Osaka 560-0043, Japan}
\email{m-sato@cr.math.sci.osaka-u.ac.jp}

\maketitle

\begin{abstract}

A hyperelliptic broken Lefschetz fibration is a generalization of a hyperelliptic Lefschetz fibration. 
We construct and compute a local signature of hyperelliptic directed broken Lefschetz fibrations
by generalizing Endo's local signature of hyperelliptic Lefschetz fibrations.
It is described by his local signature and a rational-valued homomorphism on the subgroup of
the hyperelliptic mapping class group which preserves a simple closed curve setwise.

\end{abstract}

\section{Introduction}

A broken Lefschetz fibration is a smooth map introduced in \cite{auroux2005slp} from a four-manifold to a surface  which has at most two types of singularities, called Lefschetz singularity and indefinite fold singularity. 
It can be considered as a generalization of a Lefschetz fibration,
and combining the results of Williams \cite{williams2010hbl} and Lekili \cite{lekili2009wfn},
it is proved that every closed oriented four-manifold admits a directed (more strictly, simplified) broken Lefschetz fibration.

In \cite{hayano2011fah}, 
we defined a hyperelliptic directed broken Lefschetz fibration as a generalization of a hyperelliptic Lefschetz fibration.
We showed that, when the genus of any component of any fiber is greater than or equal to two, 
after blowing up several times, the total space is a double branched covering of a manifold obtained by blowing up a sphere bundle over the sphere.
We also proved that the second rational homology class represented by a general fiber is nontrivial.
As a corollary, $\sharp n\mathbb{CP}^2$ does not admit this fibration structure.

The purpose of this paper is to investigate the homeomorphism types of hyperelliptic directed broken Lefschetz fibrations generalizing Endo's local signature in \cite{endo2000mss} of hyperelliptic Lefschetz fibrations.
See, for example, \cite{ashikaga2006vad} and \cite{ashikaga2000gal} for the history of local signatures.

We call a simple closed curve in $\Sigma_g$ is type $\I$ or type $\II_h$
if it is non-separating or separating which bounds subsurfaces of genus $h$ and $g-h$, respectively.
\begin{figure}[htbp]
\begin{center}
\includegraphics[width=90mm]{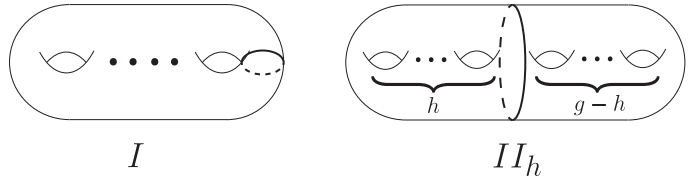}
\end{center}
\caption{type $\I$ and type $\II_h$}
\label{fig:vanishing cycles}
\end{figure}
We also call a Lefschetz singular fiber is type $\I$ or type $\II_h$
if the vanishing cycle is type $\I$ or type $\II_h$, respectively. 
Assign rational numbers to these types of singular fibers as
\[
\sigma_{\loc}(\I)=-\frac{g+1}{2g+1},\quad
\sigma_{\loc}(\II_h)=\frac{4h(g-h)}{2g+1}-1. 
\]
Endo showed that the signature of the total space of a hyperelliptic Lefschetz fibration is equal to the sum of these numbers of the Lefschetz singular fibers in the fibration (see Section \ref{section:local-signature}, for details).

To explain our main theorem, we need some notation.
Let $\Sigma_g$ be a closed oriented surface of genus $g$,
and $\mathcal{M}_g$ denote its mapping class group.
For a simple closed curve $c$ in $\Sigma_g$,
we denote by $\mathcal{M}_g(c)$ the subgroup of $\mathcal{M}_g$
which consists of mapping classes represented by diffeomorphism preserving the curve $c$ setwise.
We also denote by $\mathcal{M}_g(c^{\ori})$ the subgroup of $\mathcal{M}_g(c)$
which consists of mapping classes preserving the curve $c$ setwise and its orientation.
Let $\iota_g$ denote the involution of $\Sigma_g$ as in Figure \ref{fig:inv}. 
\begin{figure}[htbp]
\begin{center}
\includegraphics[width=110mm]{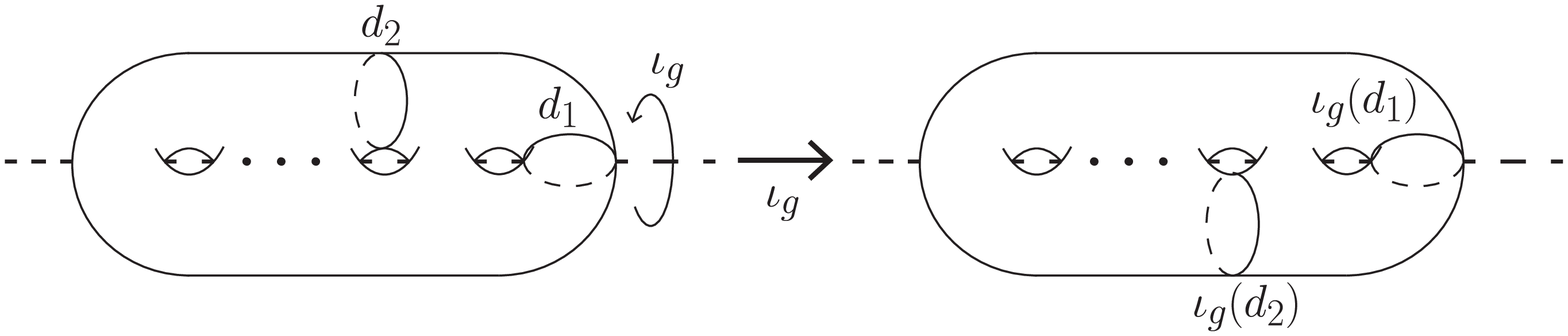}
\end{center}
\label{fig:inv}
\caption{involution $\iota_g$}
\end{figure}
We denote by $\mathcal{H}_g$ the hyperelliptic mapping class group, that is, the subgroup of $\mathcal{M}_g$
which consists of mapping classes represented by diffeomorphisms $T$ satisfying $T\iota_g=\iota_gT$.
For a simple closed curve $c$ such that $\iota_g(c)=c$, 
we also denote by $\mathcal{H}_g(c)$ and $\mathcal{H}_g(c^{\ori})$
the subgroups $\mathcal{H}_g(c)=\mathcal{H}_g\cap\mathcal{M}_g(c)$ and $\mathcal{H}_g(c^{\ori})=\mathcal{H}_g\cap\mathcal{M}_g(c^{\ori})$, respectively.
In Lemma \ref{lem:homo-h},
we will define rational-valued homomorphisms $h_{g,c}$ on $\mathcal{H}_g(c)$ when $c$ is non-separating, and on $\mathcal{H}_g(c^{\ori})$ when $c$ is separating.

Let $f:M\to S^2$ be a directed broken Lefschetz fibration. 
We denote by $Z_i$ the image of each component of the indefinite fold singularities under $f$.
Decompose the 2-sphere into annuli $A_i$ each of which is a neighborhood of $Z_i\subset S^2$ for $i=1,2,\cdots,m$,
and disks $D_l$ and $D_h$ as in Figure \ref{fig:basedecomposition}.
We may choose $D_h$ so that the image $\{y_1,\cdots,y_n\}\subset S^2$ of all the Lefschetz singularities is in $\Int D_h$.
We denote by $\partial_1 A_i$ the boundary component of $A_i$ such that $\partial_1 A_i=A_i\cap A_{i+1}$ for $i=1,\cdots,m-1$ and $\partial_1 A_m=A_m\cap D_l$.
We also denote by $\partial_0 A_i$ the other boundary component of $A_i$.
\begin{figure}[htbp]
\begin{center}
\includegraphics[width=50mm]{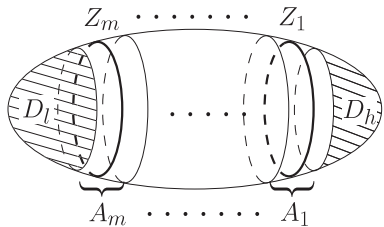}
\end{center}
\caption{annuli $A_i$ and disks $D_l$ and $D_h$ in $S^2$}
\label{fig:basedecomposition}
\end{figure}

For each $i=1,\cdots,m$,
there is a unique component $M_i$ of $f^{-1}(A_i)$ where an indefinite fold singularity exists.
Let $g_i$ denote the genus of a fiber in the mapping torus $\partial_0 A_i\cap M_i$.
Identifying the fiber with $\Sigma_{g_i}$, we consider the vanishing cycle $d_i$
of the indefinite fold singularity is in $\Sigma_{g_i}$.
We assume $f$ to be hyperelliptic, whose definition we will give in Section \ref{section:hyp-MCG}.
By the definition of the hyperelliptic directed broken Lefschetz fibration  and Lemma \ref{lem:baykur} by Baykur,
the monodromy $\varphi_i$ of the mapping torus is in $\mathcal{H}_{g_i}(d_i)$.
Then, our main theorem is as follows:
\begin{theorem}\label{thm:main theorem}
Let $f:M\to S^2$ be a hyperelliptic directed broken Lefschetz fibration as above. Then, we have
\[
\Sign M=\sum_{i=1}^{m}h_{g_i,d_i}(\varphi_i)+\sum_{j=1}^{n}\sigma_{\loc}(f^{-1}(y_j)).
\]
\end{theorem}

We will prove Theorem \ref{thm:main theorem} in Section \ref{section:localization}.
As we will see in Section \ref{section:non-additivity}, 
it is easy to calculate the explicit values of $h_{g_i,d_i}$ since it is a homomorphism.

In Section \ref{section:H_g(c)},
we will compute the abelianization and find a generating set of the groups $\mathcal{H}_g(c)$ and $\mathcal{H}_g(c^{\ori})$.
Let $c_1,c_2,\cdots,c_{2g+1}$ be simple closed curves in Figure \ref{fig:scc},
and $t_c$ denote the Dehn twist along a simple closed curve $c\subset \Sigma_g$.
\begin{figure}[htbp]
\begin{center}
\includegraphics[width=55mm]{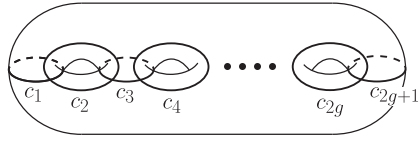}
\end{center}
\caption{simple closed curves $c_1,\cdots,c_{2g+1}$}
\label{fig:scc}
\end{figure}

\begin{proposition}\label{prop:generating set}
Let $g\ge1$.
\begin{enumerate}
\item Let $c$ be a non-separating simple closed curve of type $\I$ in Figure \ref{fig:vanishing cycles}.
The group $\mathcal{H}_g(c)$ is generated by $\{t_{c_1},\cdots,t_{c_{2g-1}},t_{c_{2g+1}},\iota_g\}$.

\item Let $1\le h\le g-1$, and $c$ a separating simple closed curve of type $\II_h$ in Figure \ref{fig:vanishing cycles}.
The group $\mathcal{H}_g(c^{\ori})$ is generated by $\{t_{c_1},t_{c_2},\cdots,t_{c_{2h}},t_{c_{2h+2}},t_{c_{2h+3}},\cdots,t_{c_{2g+1}}\}$.
\end{enumerate}
\end{proposition}

In Section \ref{section:local signature},
we will construct rational-valued homomorphisms $h_{g,c}$ on $\mathcal{H}_g(c)$ when $c$ is non-separating, and on $\mathcal{H}_g(c^{\ori})$ when $c$ is separating.
We will also compute their values on the generating sets in Proposition \ref{prop:generating set}. 
\begin{proposition}\label{prop:values of h}

\begin{enumerate}
\item Let $g\ge1$, and $c$ a non-separating simple closed curve of type $\I$ in Figure~\ref{fig:vanishing cycles}.
The values of the homomorphism $h_{g,c}:\mathcal{H}_g(c)\to\mathbb{Q}$ are
\[
h_{g,c}(\iota_g)=0,\quad
h_{g,c}(t_{c_i})=-\frac{1}{4g^2-1}\ \text{ for }i=1\cdots,2g-1,\quad\text{and }
h_{g,c}(t_{c_{2g+1}})=-\frac{g}{2g+1}.
\]

\item Let $g\ge1$, $0\le h\le g$, and $c$ a separating simple closed curve of type $\II_h$ in Figure \ref{fig:vanishing cycles}.
When $1\le h\le g-1$, the values of the homomorphism $h_{g,c}:\mathcal{H}_g(c^{\ori})\to\mathbb{Q}$ are
\begin{align*}
h_{g,c}(t_{c_i})&=\frac{g+1}{2g+1}-\frac{h+1}{2h+1} \text{\quad for }i=1,\cdots,2h,\\
h_{g,c}(t_{c_i})&=\frac{g+1}{2g+1}-\frac{g-h+1}{2(g-h)+1} \text{\quad for }i=2h+2,\cdots,2g.
\end{align*}
When $h=0,g$, the homomorphism $h_{g,c}$ is the zero map.
\end{enumerate}

\end{proposition}

In Section \ref{section:example},
we will give examples of calculations of the signatures of simplified broken Lefschetz fibrations,
and determine their homeomorphism types.

\section{Preliminary}

\subsection{Broken Lefschetz fibrations}\label{section:BLF}

\begin{definition}\label{definition:BLF}
Let $M$ and $\Sigma$ be compact oriented smooth manifolds of dimension 4 and 2, respectively.
A smooth map $f:M\to\Sigma$ is called a {\it broken Lefschetz fibration} ({\it BLF}, for short)
if it satisfies the following conditions:
\begin{enumerate}
\item $f^{-1}(\partial\Sigma)=\partial M$

\item $f$ has at most the two types of singularities which is locally written as follows:
\begin{itemize}
\item $(z_1,z_2)\mapsto \xi=z_1z_2$,
where $(z_1,z_2)$ (resp. $\xi$) is a complex local coordinate of $M$ (resp. $\Sigma$) compatible with its orientation;

\item $(t,x_1,x_2,x_3)\mapsto (y_1,y_2)=(t,x_1^2+x_2^2-x_3^2)$,
where $(t,x_1,x_2,x_3)$ (resp. $(y_1,y_2)$) is a real coordinate of $M$ (resp. $\Sigma$).
\end{itemize}
\end{enumerate}
\end{definition}

The first singularity in the condition (ii) of Definition \ref{definition:BLF} is called a {\it Lefschetz singularity} and the second one is called an {\it indefinite fold singularity}.
We denote by $\mathcal{C}_f$ the set of Lefschetz singularities of $f$ and by $Z_f$ the set of indefinite fold singularities of $f$.
We remark that a Lefschetz fibration is a BLF which has no indefinite fold singularities.

Let $f:M\rightarrow S^2$ be a BLF over the $2$-sphere. 
Suppose that the restriction of $f$ to the set of singularities is injective and that  
the image $f(Z_f)$ is the disjoint union of embedded circles parallel to the equator of $S^2$. 
We put $f(Z_f)=Z_1\amalg \cdots \amalg Z_m$, where $Z_i$ is the embedded circle in $S^2$. 
We choose a path $\alpha:[0,1]\rightarrow S^2$ satisfying the following properties: 
\begin{enumerate}

\item $\Im{\alpha}$ is contained in the complement of $f(\mathcal{C}_f)$; 

\item $\alpha$ starts at the north pole $p_h\in S^2$, and ends at the south pole $p_l\in S^2$; 

\item $\alpha$ intersects each component of $f(Z_f)$ at one point transversely. 

\end{enumerate}

\noindent
We put $\{q_i\}=Z_i\cap \Im{\alpha}$ and $\alpha(t_i)=q_i$. 
We assume that $q_1,\ldots,q_m$ appear in this order when we go along $\alpha$ from $p_h$ to $p_l$ (see Figure \ref{fig:directedBLF}). 
\begin{figure}[htbp]
\begin{center}
\includegraphics[width=50mm]{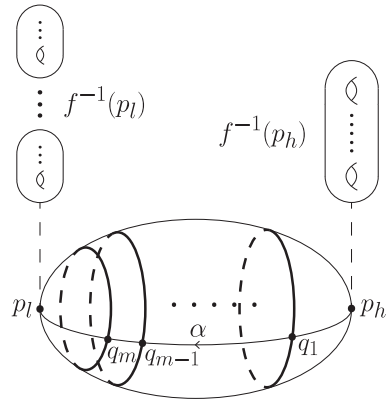}
\end{center}
\caption{The example of the path $\alpha$. 
The bold circles describe $f(Z_f)$. }
\label{fig:directedBLF}
\end{figure}
The preimage $f^{-1}(\Im{\alpha})$ is a $3$-manifold which is a cobordism between $f^{-1}(p_h)$ and $f^{-1}(p_l)$. 
By the local coordinate description of the indefinite fold singularity, 
it is easy to see that $f^{-1}(\alpha([0,t_i+\epsilon]))$ is obtained from 
$f^{-1}(\alpha([0,t_i-\epsilon]))$ by either $1$ or $2$-handle attachment for each $i=1,\ldots,m$. 
In particular, we obtain a handle decomposition of the cobordism $f^{-1}(\Im{\alpha})$. 

\begin{definition}\label{definition:DBLF}

A BLF $f:M\to S^2$ is said to be {\it directed} if it satisfies the following conditions: 
\begin{enumerate}

\item the restriction of $f$ to the set of singularities is injective and   
the image $f(Z_f)$ is the disjoint union of embedded circles parallel to the equator of $S^2$; 

\item all the handles in the above handle decomposition of $f^{-1}(\Im{\alpha})$ is index-$2$;  

\item all Lefschetz singularities of $f$ are in the preimage of the component of $S^2\setminus (Z_1\amalg\cdots\amalg Z_m)$ 
which contains the point $p_h$. 

\end{enumerate}
\end{definition}

We put $\{r_i\}=\partial_0A_i\cap \Im{\alpha}$ for $i=1,\cdots,m$ and $\{r_{m+1}\}=\partial_1A_m\cap \Im{\alpha}$.
Using the path $\alpha$, we can identify $f^{-1}(\alpha(t_i-\epsilon))$ with $f^{-1}(r_i)$. We call the attaching circle in $f^{-1}(r_i)$ of the 2-handle in Definition \ref{definition:DBLF} the {\it vanishing cycle} of $Z_i$.
A BLF is called {\it simplified} if it is directed, $m=1$, and the vanishing cycle of $Z_1$ is non-separating.

\subsection{Monodromy representations and vanishing cycles of Lefschetz singularities}\label{section:monodromy}

Let $M\to \Sigma$ be an oriented surface bundle over a 2-manifold $\Sigma$.
For a base point $y_0\in \Sigma$, we denote by $\varrho:\pi_1(\Sigma,y_0)\to \mathcal{M}_g$ the monodromy representation. 
Let $f:M\to D^2$ be a Lefschetz fibration over a disk, and let $\mathcal{C}_f=\{z_1,\cdots,z_n\}$ denote the set of Lefschetz singularities of $f$.

For each $i=1,\cdots,n$, put $y_i=f(z_i)$, and take an embedded path $\alpha_i:[0,1]\to D^2$ satisfying
\begin{itemize}
\item each $\alpha_i$ connects $y_0$ to $y_i$,
\item $\alpha_i\cap f(\mathcal{C}_f)=\{y_i\}$,
\item $\alpha_i\cap\alpha_j=\{y_0\}$ for all $i\ne j$,
\item $\alpha_1,\cdots,\alpha_n$ appear in this order when we travel counterclockwise around $y_0$.
\end{itemize}
For each $i=1,\cdots,n$,
we denote by $a_i\in\pi_1(D^2\setminus f(\mathcal{C}_f),y_0)$ the element represented by the loop obtained by connecting a counterclockwise circle around $y_i$ to $y_0$ by using $\alpha_i$.
The sequence $W_f=(\varrho_f(a_1),\cdots,\varrho_f(a_n))\in (\mathcal{M}_g)^n$ is called {\it a Hurwitz system} of $f$.
By the conditions on paths $a_1,\cdots,a_n$,
the product $\varrho_f(a_1)\cdots\varrho_f(a_n)$ is equal to the monodromy along the boundary $\partial D^2$.
It is known that each $\varrho_f(a_i)$ is the right-handed Dehn twist along a certain simple closed curve $c_i$,
called the {\it vanishing cycle} of the Lefschetz singularity $z_i$ (see \cite{kas1980hda} or \cite{matsumoto1996lfg}). 

\subsection{The hyperelliptic mapping class group and hyperelliptic directed BLFs}\label{section:hyp-MCG}

Endow the relative topology with the centralizer $C(\iota_g)$ of $\iota_g$ in the diffeomorphism group $\Diff_+\Sigma_g$.
The inclusion homomorphism $C(\iota_g)\to\Diff_+\Sigma_g$ induces a natural homomorphism $\pi_0C(\iota_g)\to\mathcal{M}_g$ between their path-connected components.
We denote this group $\pi_0C(\iota_g)$ by $\mathcal{H}_g^s$.
Birman and Hilden showed:

\begin{theorem}[Birman-Hilden {\cite[Corollary 7.1]{birman69mcc}}]\label{thm:hypMCG}
When $g\ge2$, the homomorphism $\mathcal{H}_g^s\to \mathcal{M}_g$ is injective. 
\end{theorem}

The image of the above homomorphism is called the {\it hyperelliptic mapping class group},
and denoted by $\mathcal{H}_g$. 
Actually, they showed the above result for more general settings,
but we only use the case for the involution $\iota_g:\Sigma_g\to\Sigma_g$.
See \cite{birman1973ihr}, for more details. 

Let $f:M\to S^2$ be a directed BLF.
For $i=1,\cdot,m$,
let $d_i\subset f^{-1}(r_i)$ denote the vanishing cycle of $Z_i$.
Fix an identification $f^{-1}(p_h)$ with $\Sigma_{n_1}\amalg\cdots\amalg \Sigma_{n_k}$ for some integers $n_1, \cdots, n_k$.
Then, we can define an involution of $f^{-1}(p_h)$ by $\iota_{n_1}\amalg\cdots\amalg \iota_{n_k}$.
By using the path $\alpha$, we can identify $f^{-1}(p_h)$ with $f^{-1}(r_1)$ and $f^{-1}(r_i)\setminus \{\text{two points}\}$ with $f^{-1}(r_{i+1})\setminus d_i$.
Hence, we also obtain an involution of $f^{-1}(r_i)$ by the hyperelliptic involution of $f^{-1}(p_h)$ for $i=1,\cdots,m$.

\begin{definition}\label{def:hyperelliptic}
A directed BLF $f:M\to S^2$ is said to be {\it hyperelliptic} if it satisfies the following conditions for a suitable identification of $f^{-1}(p_h)$ with $\Sigma_{n_1}\amalg\cdots\amalg \Sigma_{n_k}$:
\begin{itemize}
\item the image of the monodromy representation of the Lefschetz fibration $\res f:f^{-1}(D_h)\to D_h$ is contained in the group $\mathcal{H}_g$,

\item $d_i$ is preserved by the involution up to isotopy.
\end{itemize}
\end{definition}

In the following,
we review some properties of the hyperelliptic mapping class group. 
Let $X$ be a 2-disk or a 2-sphere.
For a positive integer $n$ and distinct points $\{p_i\}_{i=1}^n$ in $\Int X$,
Denote by $\Diff_+(X,\partial X, \{p_1,p_2,\cdots,p_n\})$ the group defined by
\begin{align*}
&\Diff_+(X,\partial X,  \{p_1,p_2,\cdots,p_n\})\\
&=\{T\in\Diff_+X\,|\,T|_{\partial X}\text{ is the identity map, and } T(\{p_1,p_2,\cdots,p_n\})=\{p_1,p_2,\cdots,p_n\}\}.
\end{align*}
Denote by $\mathcal{M}_0^n$ or $\mathcal{M}_{0,1}^n$ its mapping class group when $X=S^2$ or $X=D^2$, respectively.
Let $D_i$ be a disk in $\Int X$ which includes $p_i$ and $p_{i+1}$ but is disjoint from all $p_j$ for $j\ne i,i+1$,
and denote by $\nu(\partial D_i)$ a neighborhood of the boundary $\partial D_i$ in $D_i$. 
Choose a diffeomorphism $T_i\in \Diff_+(X, \partial X, \{p_1,p_2,\cdots,p_n\})$ such that
$T_i|_{D_i}$ interchanges the points $p_i$ and $p_{i+1}$,
$T_i|_{X-\Int D_i}$ is the identity map, 
and $T_i^2$ is isotopic to the Dehn twist along $\partial D_i$ (see Birman-Hilden p.87-88 for details).
The mapping class group $\mathcal{M}_0^n$ and $\mathcal{M}_{0,1}^n$ is generated by $\{\sigma_i\}_{i=1}^{n-1}$,
where $\sigma_i$ is the mapping class represented by the diffeomorphism $T_i$. 

Identifying the quotient space $\Sigma_g/\braket{\iota_g}$ with $S^2$,
let $\{p_1,p_2,\cdots,p_{2g+1},p_{2g+2}\}\subset S^2$ be the branched set of the quotient map $\Sigma_g\to \Sigma_g/\braket{\iota_g}$. 
By the definition, any diffeomorphism $T$ in $C(\iota_g)$ satisfies $T\iota_g(x)=\iota_g T(x)$ for $x\in\Sigma_g$.
Hence, there exists a unique diffeomorphism $\bar{T}\in \Diff_+S^2$ such that the diagram
\[
\begin{CD}
\Sigma_g@>T>> \Sigma_g\\
@VpVV@VVpV\\
S^2@>\bar{T}>> S^2
\end{CD}
\]
commutes.
Moreover, it satisfies $\bar{T}(\{p_1,p_2,\cdots,p_{2g+2}\})=\{p_1,p_2,\cdots,p_{2g+2}\}\subset S^2$.

By the above diagram, we can define
\[
\mathcal{P}_g:\mathcal{H}_g^s\to \mathcal{M}_0^{2g+2}
\]
by $\mathcal{P}_g([T])=[\bar{T}]$. 

\begin{theorem}[Birman-Hilden {\cite[Theorem 1]{birman69mcc}}]\label{thm:projection}
Let $g\ge1$. the sequence
\[
\begin{CD}
1@>>>\braket{\iota_g}@>>>\mathcal{H}_g^s @>\mathcal{P}_g>>\mathcal{M}_0^{2g+2}@>>>1
\end{CD}
\]
is exact.
\end{theorem}

They showed the homomorphism $\mathcal{P}_g:\mathcal{H}_g^s\to \mathcal{M}_0^{2g+2}$ maps
the Dehn twist $t_{c_i}$ to $\sigma_i$ in \cite[Theorem 2]{birman69mcc}. 
Furthermore, they proved:
\begin{proposition}
Let $g\ge1$. The group $\mathcal{H}_g^s$ is generated by $\{t_{c_1},\cdots,t_{c_{2g+1}}\}$.
\end{proposition}

\subsection{Meyer's signature cocycle and the local signature for hyperelliptic Lefschetz fibrations}\label{section:local-signature}

It is known that, for a hyperellitic Lefschetz fibration $f:M\to \Sigma$ over a closed oriented surface $\Sigma$,
the signature $\Sign M$ is described as the sum of invariants of the singular fiber germs in $M$.
We review this invariant.

Let $\varphi,\psi$ be elements in the mapping class group $\mathcal{M}_g$.
We denote by $E_{\varphi,\psi}$ a $\Sigma_g$-bundle over a pair of pants $S^2-\amalg_{i=1}^3 \Int D^2$ whose monodromies along $\alpha$ and $\beta$ in Figure \ref{fig:pants} are $\varphi$ and $\psi$, respectively.
\begin{figure}[htbp]
\begin{center}
\includegraphics[width=26mm]{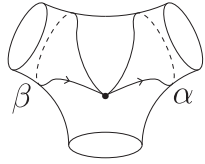}
\end{center}
\caption{paths $\alpha$ and $\beta$}
\label{fig:pants}
\end{figure}

\begin{theorem}[Meyer \cite{meyer1973sf}]
Define a $2$-cochain $\tau_g: \mathcal{M}_g\times \mathcal{M}_g\to \mathbb{Z}$ of the mapping class group by $\tau_g(\varphi,\psi)=-\Sign E_{\varphi,\psi}$. Then, $\tau_g$ is a $2$-cocycle, and the order of its homology class is as follows.
\begin{enumerate}
\item The order of $[\tau_1]\in H^2(\mathcal{M}_1;\mathbb{Z})$ is 3,
\item The order of $[\tau_2]\in H^2(\mathcal{M}_2;\mathbb{Z})$ is 5,
\item When $g\ge3$, $[\tau_g]\ne 0\in H^2(\mathcal{M}_g;\mathbb{Q})$.
\end{enumerate}
\end{theorem}

\begin{proposition}[Endo \cite{endo2000mss}]
If we restrict $\tau_g$ to $\mathcal{H}_g$, the order of $[\tau_g]\in H^2(\mathcal{H}_g;\mathbb{Z})$ is $2g+1$.
\end{proposition}

Since $\tau_g$ represents a trivial homology class in $H^2(\mathcal{H}_g;\mathbb{Q})$,
there exists a cobounding function $\phi_g:\mathcal{H}_g\to\mathbb{Q}$ of it.
Furthermore, since $H_1(\mathcal{H}_g;\mathbb{Q})$ is trivial,
this cobounding function $\phi_g$ is unique.

\begin{lemma}[Endo {\cite[Proof of Theorem 4.4]{endo2000mss}}]\label{lem:meyer}
Let $f:M\to \Sigma$ be a $\Sigma_g$-bundle over a compact oriented surface $\Sigma$. Assume that the image of the monodromy representation $\pi_1(\Sigma,y_0)\to\mathcal{M}_g$ is in $\mathcal{H}_g$ if we choose a suitable identification $f^{-1}(y_0)\cong \Sigma_g$.
Let $\{\partial_j\}_{j=1}^n$ denote the boundary components of $\Sigma$, and give orientations coming from $\Sigma$.
Then, we have
\[
\Sign M= -\sum_{j=1}^n \phi(\psi_j),
\]
where $\psi_j\in\mathcal{H}_g$ is the monodromy along $\partial_j$.
\end{lemma}

Using this function,
he generalized the local signature of Lefschetz fibrations of genus $1$ \cite{matsumoto1983mft} and of genus $2$ \cite{matsumoto1996lfg} constructed by Matsumoto.
Let $f:M\to \Sigma$ be a hyperelliptic Lefschetz fibration of genus $g$ over a closed oriented surface $\Sigma$,
and $y_1,\cdots, y_n$ the image of the set of Lefschetz singularities under $f$.
For the Lefschetz singular fiber $f^{-1}(y_j)$, define a rational number $\sigma_{\loc}(f^{-1}(y_j))$ by
\[
\sigma_{\loc}(f^{-1}(y_j))=-\phi_g(\varphi_j)+\Sign(f^{-1}\nu(y_j)),
\]
where $\varphi_j\in \mathcal{H}_g$ is the monodromy along $\partial\nu(y_j)$.
He computed the values for Lefschetz singular fibers as in Introduction, and showed:
\begin{theorem}[Endo {\cite[Theorem 4.4]{endo2000mss}}]\label{thm:local-sig}
Let $f:M\to \Sigma$ be a hyperelliptic Lefschetz fibration as above. Then, we have
\[
\Sign M=\sum_{i=1}^n\sigma_{\loc}(f^{-1}(y_j)).
\]
\end{theorem}

\section{A subgroup $\mathcal{H}_g(c)$ of the hyperelliptic mapping class group which preserves a curve $c$}\label{section:H_g(c)}
In this section, we investigate the abelianization and a generating set
of the hyperelliptic mapping class group $\mathcal{H}_g(c)$ which fix the curve $c$. 
In the last paragraphs of Section \ref{section:non-sep} and Section \ref{section:sep}, we will prove Proposition \ref{prop:generating set}.

Consider the case when $c$ is nonseparating.
If we take a diffeomorphism $T\in\Diff_+\Sigma_g$ which fixes the curve $c$ setwise,
it induces the diffeomorphism $\Sigma_g\setminus c\to\Sigma_g\setminus c$.
This diffeomorphism can be extended to the diffeomorphism of $\hat{T}$ of $\Sigma_{g-1}$
by regarding $\Sigma_g\setminus c$ as the surface of genus $g-1$ with two punctures.
Hence, we can define a homomorphism $\Phi_n:\mathcal{M}_g(c)\to \mathcal{M}_{g-1}$
by $\Phi_n([T])=[\hat{T}]$. 
Next, consider the case when $c$ is a separating curve in $\Sigma_g$
bounding subsurfaces of genus $h$ and $g-h$.
Identifying $\Sigma_g\setminus c$ with disjoint sum of two punctured surfaces of genus $h$ and $g-h$,
we can also define a homomorphism $\Phi_s:\mathcal{M}_g(c^{\ori})\to\mathcal{M}_h\times \mathcal{M}_{g-h}$.

\subsection{When $c$ is non-separating}\label{section:non-sep}

First, consider the case when $c$ is type $\I$. For symplicity,
we choose $c$ as in Figure \ref{fig:vanishing cycles}.
Let $\gamma\in \Sigma_g/\braket{\iota_g}$ be the projection of the curve $c$
by $p:\Sigma_g\to \Sigma_g/\braket{\iota_g}$.
Identifying $\Sigma_g/\braket{\iota_g}$ with $S^2$,
define a group $\mathcal{M}_0^{2g}(\gamma)$ by
\[
\mathcal{M}_0^{2g}(\gamma)=\{[T]\in\mathcal{M}_0^{2g+2}\,|\,T(\gamma)=\gamma\}.
\]
For a diffeomorphism $T\in C(\iota_g)$,
we have a diffeomorphism $\bar{T}\in\Diff_+(S^2,p_1,p_2,\cdots,p_{2g+1},p_{2g+2})$
defined by $pT=\bar{T}p$ as in Section \ref{section:hyp-MCG}.
Moreover, if $T\in C(\iota_g)$ preserves $c$ setwise,
$\bar{T}$ also preserves the path $\gamma$ setwise.
Hence, the image $\mathcal{P}_g(\mathcal{H}_g^s(c))$ is contained in $\mathcal{M}_0^{2g}(\gamma)$.
Conversely, if $\bar{T}\in\Diff_+(S^2,p_1,p_2,\cdots,p_{2g+1},p_{2g+2})$ preserves the path $\gamma$ setwise,
there is a diffeomorphism $T\in C(\iota_g)$ such that $T(c)=c$ and $pT=\bar{T}p$.
Thus, we have $\mathcal{P}_g(\mathcal{H}_g^s(c))=\mathcal{M}_0^{2g}(\gamma)$.
Consider the exact sequence obtained by restricting
the homomorphism $\mathcal{P}_g:\mathcal{H}_g^s\to\mathcal{M}_0^{2g+2}$ in Theorem \ref{thm:projection}
to $\mathcal{H}_g^s(c)$.
\begin{lemma}\label{lem:H_g^s(c)}
For $g\ge1$, the exact sequence
\[
\begin{CD}
1@>>>\mathbb{Z}/2\mathbb{Z}@>>>\mathcal{H}_g^s(c)@>\mathcal{P}_g>>\mathcal{M}_0^{2g}(\gamma)@>>>1
\end{CD}
\]
splits.
In particular, we have $\mathcal{H}_g^s(c)\cong \mathbb{Z}/2\mathbb{Z}\times\mathcal{M}_0^{2g}(\gamma)$.
\end{lemma}
\begin{proof}
Define a map $\lambda:\mathcal{H}_g^s(c)\to \mathbb{Z}/2\mathbb{Z}$ by $\lambda(\varphi)=0$
if $\varphi_*[c]=[c]\in H_1(\Sigma_g;\mathbb{Z})$,
and $\lambda(\varphi)=1$ if $\varphi_*[c]=-[c]\in H_1(\Sigma_g;\mathbb{Z})$.
Then, $\lambda$ is a homomorphism,
and satisfies $\lambda([\iota_g])=1\in\mathbb{Z}/2\mathbb{Z}$.
Thus, it induces a splitting of the exact sequence.
\end{proof}

Let $s:\partial D^2\to \partial D^2$ denote the half-rotation of the circle. 
Let $\mathcal{M}_{0,\half}^{2g}$ denote the group which consists of the path-connected components of
$\{T\in \Diff_+(D^2,p_1,p_2,\cdots,p_{2g})\,|\,T|_{\partial D^2}=s\text{ or }\id_{\partial D^2}\}$.

\begin{lemma}\label{lem:gamma half}
Let $g\ge1$.
\[
\mathcal{M}_0^{2g}(\gamma)\cong\mathcal{M}_{0,\half}^{2g}.
\]
\end{lemma} 
\begin{proof}
Let $\mathcal{M}_0^{2g}(\gamma^\text{ori})$ be a subgroup of $\mathcal{M}_0^{2g}(\gamma)$
consists of mapping classes which preserve the orientation of the path $\gamma$.
First, we prove the isomorphism
\[
\mathcal{M}_0^{2g}(\gamma^\text{ori})\cong\mathcal{M}_{0,1}^{2g}.
\]

Let $\Diff_+(S^2, \{p_1,\cdots,p_{2g+2}\},[\gamma])$ be the group
consists of orientation-preserving diffeomorphisms $T:S^2\to S^2$
such that $T(\{p_1,\cdots,p_{2g+2}\})=\{p_1,\cdots,p_{2g+2}\}$
and there exists a closed neighborhood $\nu(\gamma)$ of $\gamma$ where $T|_{\nu(\gamma)}$ is the identity map.
Let $T$ be a representative of a mapping class in $\mathcal{M}^{2g}(\gamma^{\ori})$.
Using the isotopy extension theorem,
we can change $T$ into a diffeomorphism in $\Diff_+(S^2, \{p_1,\cdots,p_{2g+2}\},[\gamma])$ by some isotopy.
Moreover, we can also prove that
\[
\mathcal{M}^{2g}(\gamma^{\ori})\cong \pi_0\Diff_+(S^2, \{p_1,\cdots,p_{2g+2}\},[\gamma]),
\]
using the isotopy extension theorem.
Similarly, let $\Diff_+(S^2-\Int D^2, p_1,\cdots,p_{2g}, [\partial D^2])$ be a group
consists of orientation-preserving diffeomorphisms $T:S^2-\Int D^2\to S^2-\Int D^2$
such that there exists a closed neighborhood $\nu(\partial D^2)$
where $T|_{\nu(\partial D^2)}$ is the identity map.
We can also show that
\[
\mathcal{M}_{0,1}^{2g}\cong \pi_0\Diff_+(S^2-\Int D^2, p_1,\cdots,p_{2g}, [\partial D^2]).
\]
Separate the circle $\partial D^2$ into two arcs $\alpha:[0,1]\to \partial D^2$ and $\beta:[0,1]\to \partial D^2$
such that $\alpha(0)=\beta(0)$ and $\alpha(1)=\beta(1)$.
If we identify $\alpha(t)$ and $\beta(t)$ in $S^2-\Int D^2$, the quotient space is diffeomorphic to $S^2$. 
Choose an identification $L$ of the $(2g+3)$-tuples
\[
(S^2-\Int D^2/(\alpha(t)\sim \beta(t)), p_1,\cdots,p_{2g},\alpha(0),\alpha(1))
\cong (S^2, p_1,\cdots,p_{2g},p_{2g+1},p_{2g+2}).
\]
Since a diffeomorphism $T\in \Diff_+(S^2-\Int D^2)$ satisfying $T|_{\nu(\partial D^2)}=\id_{\nu(\partial D^2)}$ 
induces a diffeomorphism $\bar{T}$ of $S^2-\Int D^2/(\alpha(t)\sim \beta(t))$,
we have the isomorphism $\mathcal{M}_{0,1}^{2g}\cong \mathcal{M}^{2g}(\gamma^{\ori})$
defined by $[T]\mapsto [L\bar{T}L^{-1}]$.

Next, we prove $\mathcal{M}_0^{2g}(\gamma)\cong\mathcal{M}_{0,\half}^{2g}$.
Choose a diffeomorphism $r\in \Diff_+(S^2-\Int D^2)$ such that
$r\alpha(t)=\beta(1-t)$ and $r(\{p_1,\cdots,p_{2g}\})=\{p_1,\cdots,p_{2g}\}$.
It induces a diffeomorphism $\bar{r}\in \Diff_+ S^2$ such that
$\bar{r}(\{p_1,\cdots,p_{2g}\})=\{p_1,\cdots,p_{2g}\}$, $\bar{r}(p_{2g+1})=p_{2g+2}$,
and $\bar{r}(p_{2g+2})=p_{2g+1}$.
Consider the group consisting of diffeomorphisms $T$ of $S^2$ such that
$T(\{p_1,\cdots,p_{2g+2}\})=\{p_1,\cdots,p_{2g+2}\}$, and
$T|_{\nu(\gamma)}$ is equal to $\bar{r}|_{\nu(\gamma)}$ \text{ or } $\id_{\nu(\gamma)}$
for some closed neighborhood $\nu(\gamma)$
instead of $\Diff_+(S^2, \{p_1,\cdots,p_{2g+2}\},[\gamma])$.
In the same way, 
consider the group consisting of diffeomorphisms $T$ of $S^2-\Int D^2$ such that
$T(\{p_1,\cdots,p_{2g}\})=\{p_1,\cdots,p_{2g}\}$, and
$T|_{\nu(\partial D^2)}$ is equal to $r|_{\nu(\partial D^2)}$ \text{ or } $\id_{\nu(\partial D^2)}$
instead of $\Diff_+(S^2-\Int D^2, p_1,\cdots,p_{2g}, [\partial D^2])$.
Then, we have the isomorphism between their path-connected components, similarly.
Thus, we have $\mathcal{M}^{2g}(\gamma)\cong \mathcal{M}_{0,\half}^{2g}$.
\end{proof}

We can define a homomorphism $\mathcal{M}_{0,\half}^{2g}\to\braket{s}$ by mapping $[T]$ to $T|_{\partial D^2}$,
where $\braket{s}$ is the cyclic group of order 2 generated by $s$. 
Then, the kernel is the subgroup $\mathcal{M}_{0,1}^{2g}$.
\begin{lemma}\label{lem:M_{0,half}}
For $g\ge1$, the exact sequence
\[
\begin{CD}
1@>>>\mathcal{M}_{0,1}^{2g}@>>>\mathcal{M}_{0,\half}^{2g}@>>>\mathbb{Z}/2\mathbb{Z}@>>>1
\end{CD}
\]
splits. 
\end{lemma}
\begin{proof}
We may assume $p_1,\cdots,p_{2g}$ are arranged in the disk as in Figure \ref{fig:split}. 
Consider an involution $\mu\in \Diff_+(D^2, p_1,\cdots,p_{2g})$ which rotates the disk 180 degrees
and interchanges the points $p_i$ and $p_{g+i}$ for $i=1,\cdots,g$.
\begin{figure}[htbp]
\begin{center}
\includegraphics[width=25mm]{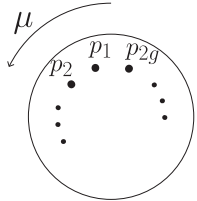}
\end{center}
\caption{$p_1,\cdots,p_{2g}$ in $D^2$}
\label{fig:split}
\end{figure}
Define a homomorphism $j:\mathbb{Z}/2\mathbb{Z}\to \mathcal{M}_{0,\half}^{2g}$ by $j(1)=\mu$.
This induces the splitting of the above exact sequence.
\end{proof}
\begin{lemma}\label{lem:first homology}
Let $g\ge1$, and $c$ a non-separating simple closed curve such that $\iota_g(c)=c$.
Then, we have
\[
H_1(\mathcal{H}_g^s(c);\mathbb{Z})=\mathbb{Z}\oplus(\mathbb{Z}/2\mathbb{Z})^2.
\]
\end{lemma}
\begin{proof}
By Lemma \ref{lem:H_g^s(c)} and Lemma \ref{lem:M_{0,half}}, we have
\[
H_1(\mathcal{H}_g^s(c);\mathbb{Z})\cong \mathbb{Z}/2\mathbb{Z}\oplus H_1(\mathcal{M}_0^{2g}(\gamma);\mathbb{Z}),
\text{ and }
H_1(\mathcal{M}_{0,\half}^{2g};\mathbb{Z})\cong H_1(\mathcal{M}_{0,1}^{2g};\mathbb{Z})\oplus \mathbb{Z}/2\mathbb{Z}.
\]
We showed $\mathcal{M}_0^{2g}(\gamma)\cong \mathcal{M}_{0,\half}^{2g}$ in Lemma \ref{lem:gamma half},
and it is known that $H_1(\mathcal{M}_{0,1}^{2g};\mathbb{Z})\cong \mathbb{Z}$
(see, for example, \cite[Section 9.1.3 and 9.2]{farb2011pmcg}).
Hence, we have $H_1(\mathcal{H}_g^s(c);\mathbb{Z})\cong \mathbb{Z}\oplus (\mathbb{Z}/2\mathbb{Z})^2$.
\end{proof}

Define a group $\mathcal{H}_g^s(c)$ by $\mathcal{H}_g^s(c)=\{[T]\in \mathcal{H}_g^s\,|\, T(c)=c\}$.
If we restrict the homomorphism $\mathcal{H}_g^s\to\mathcal{H}_g$
in Theorem \ref{thm:hypMCG} to $\mathcal{H}_g^s(c)$,
we have a homomorphism $\mathcal{H}_g^s(c)\to\mathcal{H}_g(c)$.
Note that it is not obvious that this homomorphism is surjective,
in other words, mapping classes in $\mathcal{H}_g(c)$ can be represented
by elements in $C(\iota_g)$ which fix the curve $c$ setwise.
In \cite[Lemma 3.1]{hayano2011fah}, we showed:
\begin{lemma}\label{lem:isomH_g(c)}
Let $g\ge1$, and $c$ an essential simple closed curve in $\Sigma_g$.
The homomorphism
\[
\mathcal{H}_g^s(c)\to \mathcal{H}_g(c)
\]
is surjective.
\end{lemma}
By Theorem \ref{thm:hypMCG}, this is also injective when $g\ge2$.

Consider the case when $g=1$.
As is well-known, the group $\mathcal{H}_1$ coincides with $\mathcal{M}_1$.
Hence, $\mathcal{H}_1(c)$ also coincides with $\mathcal{M}_1(c)$.
If $c=c_3$ in Figure \ref{fig:scc},
the group $\mathcal{M}_1(c)$ is described as
\[
\mathcal{M}_1(c)=
\left\{\left.
\begin{pmatrix}
\epsilon&n\\
0&\epsilon
\end{pmatrix}
\in \SL(2;\mathbb{Z})
\,\right|\,\epsilon\in\{\pm1\},n\in\mathbb{Z}
\right\}.
\]
By mapping $[T]\in\mathcal{M}_1(c)$ to $\epsilon\in\mathbb{Z}/2\mathbb{Z}$,
we have a split exact sequence
\[
\begin{CD}
1@>>>\mathbb{Z}@>>>\mathcal{M}_1(c)@>>>\mathbb{Z}/2\mathbb{Z}@>>>1.
\end{CD}
\]
Thus, we have $H_1(\mathcal{H}_1(c);\mathbb{Z})=\mathbb{Z}\oplus\mathbb{Z}/2\mathbb{Z}$.
Combining Lemma \ref{lem:first homology}, Lemma \ref{lem:isomH_g(c)}, and the case when $g=1$ as above, we have:
\begin{lemma}
Let $c$ be a non-separating simple closed curve such that $\iota_g(c)=c$. Then, we have
\[
H_1(\mathcal{H}_g(c);\mathbb{Z})
=
\begin{cases}
\mathbb{Z}\oplus(\mathbb{Z}/2\mathbb{Z})^2& \text{ when }g\ge2,\\
\mathbb{Z}\oplus\mathbb{Z}/2\mathbb{Z}& \text{ when }g=1.
\end{cases}
\]
\end{lemma}
\begin{proof}[Proof of Proposition \ref{prop:generating set} (i)]
Let $\sigma\in\mathcal{M}_{0,\half}^{2g}$ denote the half twist along $\partial D^2$.
By the exact sequence in Lemma \ref{lem:M_{0,half}},
the group $\mathcal{M}_{0,\half}^{2g}$ is generated by $\{\sigma_1,\cdots,\sigma_{2g-1},\sigma\}$.
By \cite[Theorem 2]{birman69mcc},
we have $\mathcal{P}_g(t_{c_i})=\sigma_i$ for $i=1,\cdots,2g$ and $\mathcal{P}_g(t_{c_{2g+1}})=\sigma$.
By the exact sequence in Lemma \ref{lem:H_g^s(c)},
the group $\mathcal{H}_g(c)$ is generated by $t_{c_i}$ for $i=1,2,\cdots,2g-1, 2g+1$ and $\iota_g$.
\end{proof}

\subsection{When $c$ is separating}\label{section:sep}

Next, consider the case when $c$ is type $\II_h$. For symplicity, we choose $c$ as in Figure \ref{fig:vanishing cycles}.

As we will see in Section \ref{section:round cobordisms},
when the vanishing cycle of $Z_i$ in the hyperelliptic directed BLF is separating,
the image of the monodromy representation along $\partial_0 A_i$ is contained in $\mathcal{H}_g(c^{\ori})$.
Hence, we only consider the group $\mathcal{H}_g(c^{\ori})$ in this section instead of $\mathcal{H}_g(c)$.
Of course, if $g\ne2h$, we have $\mathcal{H}_g(c)=\mathcal{H}_g(c^{\ori})$ since any diffeomorphism of $\Sigma_g$ which preserves $c$ setwise acts trivially on $\pi_0(\Sigma_g-c)$.

First, consider the case when $h=0,g$.
For any diffeomorphism $T$ of $\Sigma_g$,
we can change $T$ so that it preserves $c$ setwise by some isotopy.
Thus, we have $\mathcal{H}_g(c^{\ori})=\mathcal{H}_g$.

In the following, we only consider the case $1\le h\le g-1$.
Choose a disk $D$ in $\Sigma_g-\bigcup_{i=1}^{2g}c_i$ so that $\iota_g(D)=D$, where $c_i$ is the simple closed curve in Figure \ref{fig:scc}.
Denote by $\Sigma_{g,1}$ the subsurface $\Sigma_g-\Int D$, and by $\iota_{g,1}$ the restriction of $\iota_g$ to $\Sigma_{g,1}$.
The mapping class group $\mathcal{M}_{g,1}$ of $\Sigma_{g,1}$ is defined by $\mathcal{M}_{g,1}=\pi_0\Diff_+(\Sigma_{g,1},\partial\Sigma_{g,1})$, 
where $\Diff_+(\Sigma_{g,1},\partial\Sigma_{g,1})$ is the diffeomorphism group of $\Sigma_{g,1}$ with $C^\infty$ topology which fixes the boundary pointwise.

We identify the subsurfaces of $\Sigma_g$ bounded by $c$ with $\Sigma_{h,1}$ and $\Sigma_{g-h,1}$
so that $\iota_g|_{\Sigma_{h,1}}=\iota_{h,1}$ and $\iota_g|_{\Sigma_{g-h,1}}=\iota_{g-h,1}$.
For $T_1\in \Diff_+(\Sigma_{h,1},\partial\Sigma_{h,1})$ and $T_2\in \Diff_+(\Sigma_{g-h,1},\partial\Sigma_{g-h,1})$, the diffeomorphism $T_1\cup T_2\in\Diff_+\Sigma_g$ preserves the curve $c$. Hence, we can define a map
\[
\Psi:\mathcal{M}_{h,1}\times \mathcal{M}_{g-h,1}\to\mathcal{M}_g(c^{\ori})
\]
by $\Psi([T_1],[T_2])=[T_1\cup T_2]$. This is a well-defined homomorphism.

Define a subgroup $\mathcal{H}_{g,1}$ of $\mathcal{M}_{g,1}$ by $\mathcal{H}_{g,1}=\{[T]\in\mathcal{M}_{g,1}\,|\,\iota_{g,1}T\iota_{g,1}^{-1}=T\}$. Apparently, the image $\Psi(\mathcal{H}_{h,1}\times \mathcal{H}_{g-h,1})$ is contained in the subgroup $\mathcal{H}_g(c^{\ori})\subset \mathcal{M}_g(c^{\ori})$.

\begin{lemma}\label{lem:sep-generator}
Let $g\ge2$.
When $1\le h\le g-1$, the sequence
\[
\begin{CD}
1@>>>\mathbb{Z}@>>>\mathcal{H}_{h,1}\times\mathcal{H}_{g-h,1}@>\Psi>>\mathcal{H}_g(c^{\ori})@>>>1
\end{CD}
\]
is exact.
\end{lemma}

\begin{proof}
By \cite[Theorem 3.18]{farb2011pmcg}, we have
\[
\begin{CD}
1@>>>\mathbb{Z}@>>>\mathcal{M}_{h,1}\times\mathcal{M}_{g-h,1}@>\Psi>>\mathcal{M}_g(c^{\ori})@>>>1.
\end{CD}
\]
The kernel of $\Psi$ is generated by $(t_{\partial\Sigma_1}, t_{\partial\Sigma_2}^{-1})$,
and it is contained in $\mathcal{H}_{h,1}\times\mathcal{H}_{g-h,1}$.
Thus, we only need to prove $\Psi(\mathcal{H}_{h,1}\times\mathcal{H}_{g-h,1})=\mathcal{H}_g(c^{\ori})$.

Let $\varphi$ be a mapping class in $\mathcal{H}_g(c)$.
By Lemma \ref{lem:isomH_g(c)},
we can choose a representative $T\in\Diff_+\Sigma_g$ of $\varphi$
satisfying $T\iota_g=\iota_g T$ and $T(c)=c$.
Using some isotopy, we may assume $T|_{c}$ is the identity map.
Then, $T|_{\Sigma_{h,1}}$ and $T|_{\Sigma_{g-h,1}}$ represent mapping classes
in $\mathcal{H}_{h,1}$ and $\mathcal{H}_{g-h,1}$, respectively. 
Since $\Psi([T|_{\Sigma_{h,1}}], [T|_{\Sigma_{g-h,1}}])=[T]$,
we obtain $\Psi(\mathcal{H}_{h,1}\times\mathcal{H}_{g-h,1})=\mathcal{H}_g(c^{\ori})$.
\end{proof}

Let $C(\iota_{g,1})$ be the group defined by
$C(\iota_{g,1})=\{T\in\Diff_+(\Sigma_{g,1},\partial\Sigma_{g,1})\,|\, \iota_{g,1}T\iota_{g,1}^{-1}=T\}$.
We have the homomorphism $\mathcal{P}_{g,1}:\pi_0(C(\iota_{g,1}))\to \mathcal{M}_{0,1}^{2g+1}$
defined by $[T]\mapsto [\bar{T}]$ in the same way as $\mathcal{P}_g:\mathcal{H}_g^s\to\mathcal{M}_0^{2g+2}$ in Section \ref{section:hyp-MCG}.
Since any isotopy of $\Diff_+(D^2,\partial D^2,\{p_1,\cdots,p_{2g+1}\})$ can be lifted to an isotopy of $C(\iota_{g,1})$, $\Ker(\mathcal{P}_{g,1})$ is represented by the deck transformation $\iota_{g,1}$ or $\id_{\Sigma_{g,1}}$.
Since $C(\iota_{g,1})$ does not contain $\iota_{g,1}$,
the kernel of the homomorphism $\mathcal{P}_{g,1}$ is trivial. 
Furthermore, $\mathcal{P}_{g,1}: \pi_0C(\iota_{g,1})\to \mathcal{M}_{0,1}^{2g+1}$ is an isomorphism since $\mathcal{M}_{0,1}^{2g+1}$ is generated by $\{\sigma_i\}_{i=1}^{2g}$ and $\mathcal{P}_{g,1}(t_{c_i})=\sigma_i$ for $i=1,\cdots,2g$.

\begin{lemma}\label{lem:H_{g,1}-isom}
For $g\ge1$, the natural homomorphism $\pi_0C(\iota_{g,1})\to \mathcal{H}_{g,1}$ is an isomorphism.
\end{lemma}

\begin{proof}
By the definition of $\mathcal{H}_{g,1}$,
the natural homomorphism $\pi_0(C(\iota_{g,1}))\to \mathcal{H}_{g,1}$ is surjective.
Hence, it suffices to show the injectivity.

Embed $\Sigma_{g,1}$ in $\Sigma_{g+1}$ so that $\iota_{g+1}|_{\Sigma_{g,1}}=\iota_{g,1}$.
For a diffeomorphism $T$ of $\Sigma_{g,1}$, we can extend $T$ to a diffeomorphism $\tilde{T}$ of $\Sigma_{g+1}$ by the identity map on $\Sigma_{g+1}\setminus\Sigma_{g,1}$.
Thus, we have homomorphisms $\pi_0(C(\iota_{g,1}))\to \pi_0(C(\iota_{g+1}))$ and $\mathcal{H}_{g,1}\to\mathcal{H}_{g+1}$ defined by $[T] \mapsto [\tilde{T}]$.
By gluing a disk with three marked points to $D^2$,
we can also define a homomorphism $\mathcal{M}_{0,1}^{2g+1}\to\mathcal{M}_{0}^{2g+4}$ in the same way.
By  in \cite[Theorem 3.18]{farb2011pmcg},
the latter homomorphism is injective.

If we consider $(\Sigma_{g+1}\setminus\Int \Sigma_{g,1})/\braket{\iota_{g+1}}$ as a disk with three marked points,
we have a commutative diagram
\[
\begin{CD}
\mathcal{M}_{0,1}^{2g+1}@<\mathcal{P}_{g,1}<<\pi_0C(\iota_{g,1})@>>>\mathcal{H}_{g,1}\\
@VVV@VVV@VVV\\
\mathcal{M}_0^{2g+4}@<\mathcal{P}_{g+1}<<\pi_0C(\iota_{g+1})@>>>\mathcal{H}_{g+1}.
\end{CD}
\]
The left side shows that $\pi_0C(\iota_{g,1})\to \pi_0C(\iota_{g+1})$ is injective. 
By Theorem \ref{thm:hypMCG}, the right side shows that $\pi_0C(\iota_{g,1})\to\mathcal{H}_{g,1}$ is also injective.
\end{proof}

\begin{lemma}\label{lem:abelianization sep}
Let $g\ge2$ and $1\le h\le g-1$.
Let $c$ be a separating simple closed curve which bounds subsurfaces of genus $h$ and $g-h$ and satisfies $\iota_g(c)=c$. 
Then, we have
\[
H_1(\mathcal{H}_g(c^{\ori});\mathbb{Z})\cong \mathbb{Z}\oplus \mathbb{Z}/d\mathbb{Z},
\]
where $d=\gcd(4h(2h+1), 4(g-h)(2g-2h+1))$. 
\end{lemma}

\begin{proof}
Since $\mathcal{H}_{h,1}\cong \mathcal{M}_{0,1}^{2h+1}$,
we have $H_1(\mathcal{H}_{h,1};\mathbb{Z})\cong\mathbb{Z}$.
By the chain relation (see, for example, in \cite[Proposition 4.12]{farb2011pmcg}),
the mapping class $(t_{c_1}\cdots t_{c_{2h}})^{4h+2}\in\mathcal{H}_{h,1}$
coincides with the Dehn twist $t_{\partial \Sigma_{h,1}}$ along the boundary.
In the same way,
we have $(t_{c_1}\cdots t_{c_{2(g-h)}})^{4(g-h)+2}=t_{\partial \Sigma_{g-h,1}}\in\mathcal{H}_{g-h,1}$.

The kernel of the homomorphism $\mathcal{H}_{h,1}\times\mathcal{H}_{g-h,1}\to\mathcal{H}_g(c^{\ori})$ is
the cyclic group generated by $(t_{\partial\Sigma_{h,1}},t_{\partial\Sigma_{g-h,1}}^{-1})$.
Hence, we have
\[
H_1(\mathcal{H}_g(c^{\ori});\mathbb{Z})\cong(\mathbb{Z}\oplus \mathbb{Z})/\braket{(4h(2h+1), -4(g-h)(2g-2h+1))}\cong\mathbb{Z}\oplus \mathbb{Z}/d\mathbb{Z}.
\]
\end{proof}

\begin{proof}[Proof of Proposition \ref{prop:generating set} (ii)]
As explained in the paragraph before Lemma \ref{lem:H_{g,1}-isom},
$\mathcal{H}_{g,1}$ is generated by $t_{c_1},\cdots, t_{c_{2g}}$.
Thus, $\mathcal{H}_g(c)$ is generated by
$t_{c_1}, t_{c_2}, \cdots t_{c_{2h}}, t_{c_{2h+2}}, t_{c_{2h+3}},\cdots, t_{c_{2g+2}}$
by Lemma \ref{lem:sep-generator}.
\end{proof}

\section{Localization of the signature of directed BLFs}\label{section:local signature}

In this section, we compute the signature of $f^{-1}(A_i)$,
and show that the signature of hyperelliptic directed BLF localizes.
In Section \ref{section:round cobordisms} and Section \ref{section:non-additivity},
we calculate the signature of round cobordisms.
In Section \ref{section:homomorphism h_{g,c}}, we define a homomorphism $h_{g,c}$ for a simple closed curve $c$,
and prove Proposition \ref{prop:values of h}.
In Section \ref{section:localization}, we will prove Theorem \ref{section:localization}.

\subsection{Lemmas on round cobordisms}\label{section:round cobordisms}

We use the notation in Introduction and Section \ref{section:BLF}. 
Let $f: M\to S^2$ be a directed BLF.
In \cite[Lemma 2.2]{baykur2007tbl},
Baykur observed that $M_i$ in a directed BLF is obtained by gluing a round 2-handle to a surface bundle over an annulus.
We review his observation and investigate the signature of $f^{-1}(A_i)$.

\begin{lemma}[Baykur \cite{baykur2007tbl}]\label{lem:baykur}
Let $f:M\to S^2$ be a directed BLF.
Identifying $f^{-1}(r_i)\cap M_i$ with the surface $\Sigma_{g_i}$,
consider the vanishing cycle $d_i$ of $Z_i$ is in $\Sigma_{g_i}$.
Then, the monodromy $\varphi$ of $f^{-1}(r_i)\cap M_i$ along $\partial_0 A_i$ is in $\mathcal{M}_{g_i}(d_i)$.
\end{lemma}

Moreover, when $d_i$ is a separating curve, the monodromy $\varphi$ is in $\mathcal{M}_{g_i}(d_i^{\ori})$.
This is because, if $\varphi$ changes the orientation of $d_i$,
the monodromy along $\partial_1 A_i$ permutes the component of $f^{-1}(\gamma_{i+1})$.
Inductively, the monodromy along $\partial_1 A_m$ permutes the component of $f^{-1}(\gamma_{m+1})$.
However, since $f^{-1}(D_l)$ is a trivial surface bundle over a disk, it must be trivial. 

In the paragraph after Lemma \ref{lem:isomH_g(c)},
we defined the homomorphisms $\Phi_n: \mathcal{M}_{g_i}(c)\to \mathcal{M}_{g_i-1}$ and $\Phi_s: \mathcal{M}_{g_i}(c^{\ori})\to \mathcal{M}_h\times\mathcal{M}_{g_i-h}$.
Baykur \cite{baykur2007tbl} also observed that the monodromy of $f^{-1}(r_i)\cap M_i$ along $\partial_1 A_i$ is $\Phi_n(\varphi)$ when $d_i$ is non-separating and $\Phi_s(\varphi)$ when $d_i$ is separating, for a suitable identification of $f^{-1}(r_{i+1})\cap M_{i}$ with $\Sigma_{g_i-1}$ if $d_i$ is type $\I$ and with $\Sigma_h \cup \Sigma_{g_i-h}$ if $d_i$ is type $\II_h$.

For a mapping class $\varphi\in\mathcal{M}_g(c)$ represented by $T\in\Diff_+\Sigma_g$ satisfying $T(c)=c$,
define a mapping torus $V_\varphi$ by $V_\varphi=\Sigma_g\times [0,1]/((0,T(x))\sim (1,x))$.
We can identify $f^{-1}(\partial_0 A_i)\cap M_i$ with $V_\varphi$ for some $\varphi\in\mathcal{M}_{g_i}(d_i)$.
Identifying $D^2$ with the unit disk in $\mathbf{C}$,
define an equivalence relation $\sim_\varphi$ on $D^2\times [-1,1]\times [0,1]$ by
\[
\begin{cases}
(v,s,1)\sim (v,s,0),&\text{ if }\varphi\text{ preserves the orientation of }c,\\
(v,s,1)\sim (\bar{v},-s,0),&\text{ if }\varphi\text{ reverses the orientation of }c,
\end{cases}
\]
where $\bar{v}$ is the complex conjugate of $v$.
The compact 4-manifold $R$ defined by $R=D^2\times [-1,1]\times [0,1]/\sim_\varphi$ is called a round 2-handle. 
Choose an embedding $j:\partial D^2\times [-1,1]\times [0,1]/\sim_\varphi\to V_{\varphi}$ such that $j(\partial D^2,s,0)=c\times\{0\}\subset V_{\varphi}$ and $p_2 j(x,s,t)=t$,
where $p_2$ is the projection to the second factor.
Then, he observed:
\begin{lemma}[Baykur \cite{baykur2007tbl}]\label{lem:M_i}
\[
M_i\cong (V_\varphi\times[0,1])\cup R,
\]
for some embedding $j$ as above. 
\end{lemma}

We remark that the isotopy class of the attaching map $j:\partial D^2\times [-1,1]\times [0,1]/\sim_T\to V_\varphi\times \{0\}$ is unique if the genus $g$ is greater than or equal to 2.

The signature of $f^{-1}(A_i)$ is calculated as follows.
Since the components of $f^{-1}(A_i)$ except $M_i$ are surface bundles over the annulus $A_i$,
we have $\Sign f^{-1}(A_i)=\Sign M_i$.
By Lemma \ref{lem:M_i}, we have:

\begin{lemma}\label{lem:signature}
\[
\Sign f^{-1}(A_i)=\Sign((V_\varphi\times[0,1])\cup R).
\]
\end{lemma}

\subsection{Wall's non-additivity formula}\label{section:non-additivity}

In \cite{sato2008cfm}, the second author defined a class function $m:\mathcal{M}_{g,2}\to \mathbb{QP}^1$.
We review this function,
and calculate the signature of the compact 4-manifold $(V_\varphi\times[0,1])\cup R$ in Section \ref{section:round cobordisms}. 

For a mapping class $\varphi=[T]\in\mathcal{M}_{g,2}$,
let $V_\varphi'=\Sigma_{g,2}\times[0,1]/(0,T(x))\sim(1,x)$ be its mapping torus.
Choose points $x_1$ and $x_2$ in each boundary component of $\Sigma_{g,2}$,
and define a continuous map by $l_i:S^1\to V_\varphi'$ by $l_i(t)=(t,x_i)$ for $i=1,2$.
Let $\partial_1$ and $\partial_2$ be the two boundary components of $\Sigma_{g,2}$.
Denote by $e_1$, $e_2$, $e_3$, and $e_4$ the homology classes $[l_1]$, $[l_2]$,
$[\partial_1\times\{0\}]$, and $[\partial_2\times\{0\}]$, respectively.
Then, for some $p,q\in\mathbb{Q}$, the set $\{e_1+e_2,p(e_3-e_4)+qe_1\}$ forms a basis of $\Ker(H_1(\partial V_\varphi';\mathbb{Q})\to H_1(V_\varphi';\mathbb{Q}))$.
The element $[p:q]\in\mathbb{QP}^1$ is unique, and we can define a function $m:\mathcal{M}_{g,2}\to\mathbb{QP}^1$ by $m(\varphi)=[p:q]$.
Since it satisfies $m(\varphi t_{\partial_1}t_{\partial_2}^{-1})=m(\varphi)$, it induces the class function on $\mathcal{M}_{g}(c^{\ori})$.
For simplicity, we also denote it by $m:\mathcal{M}_{g}(c^{\ori})\to \mathbb{QP}^1$.

Define a map $s:\mathcal{M}_g(c)\to \mathbb{Z}$ by $s(\varphi)=\Sign((V_\varphi\times[0,1])\cup R)$. We can write the signature $s(\varphi)$ with the function $m:\mathcal{M}_{g}(c^{\ori})\to \mathbb{QP}^1$ as follows:
\begin{lemma}\label{lem:signature round-cob}
Let $\varphi\in\mathcal{M}_g(c)$. Then, we have
\[s(\varphi)=
\begin{cases}
\sign(m(\varphi)),&\text{ if }c \text{ is non-separating and } \varphi \text{ preserves the orientation of } c,\\
0,&\text{ otherwise}.
\end{cases}
\]
\end{lemma}

\begin{proof}
We apply Wall's nonadditivity Formula to the pasting of the round 2-handle.
First, we review his formula.
Let $X_{-}$, $X_0$, and $X_+$ be compact 3-manifolds, and let $Y_-$ and $Y_+$ be compact 4-manifolds such that
\[
\partial X_-=\partial X_+=\partial X_+=Z,\quad \partial Y_-=X_-\cup X_0,\quad \partial Y_+=X_+\cup X_0.
\]
We denote by $Y$ and $X$ the compact 4-manifold $Y=Y_-\cup Y_+$ and the space $X=X_-\cup X_0\cup X_+$, respectively. Suppose that $Y$ is oriented inducing orientations of $Y_-$ and $Y_+$. Orient the other manifolds so that
\[
\partial_*[Y_-]=[X_0]-[X_-],\quad \partial_*[Y_+]=[X_+]-[X_0],\quad \partial_*[X_-]=\partial_*[X_+]=\partial_*[X_0]=[Z].
\]

Let $V$, $A$, $B$, and $C$ denote the vector spaces $V=H_1(Z;\mathbb{Q})$, $A=\Ker(H_1(Z;\mathbb{Q})\to H_1(X_-;\mathbb{Q}))$, $B=\Ker(H_1(Z;\mathbb{Q})\to H_1(X_0;\mathbb{Q}))$, and $C=\Ker(H_1(Z;\mathbb{Q})\to H_1(X_+;\mathbb{Q}))$. 
On the vector space $W=B\cap(C+A)/((B\cap C)+(B\cap A))$,
Wall defined a symmetric bilinear map $\Psi:W\times W\to \mathbb{Q}$ as follows.
Let $I:H_1(Z;\mathbb{Q})\times H_1(Z;\mathbb{Q})\to \mathbb{Q}$ denote the intersection form, and $b,b'\in B\cap(C+A)$.
Since $b'\in B\cap(C+A)$, there exist $c'\in C$ and $a'\in A$ such that $a'+b'+c'=0$.
Then, define a map $\Psi:W\times W\to \mathbb{Q}$ by $\Psi'([b],[b'])=I(b,c')$.
He showed that this map is well-defined and symmetric.
Denote by $\Sign(V;B,C,A)$ the signature of this symmetric bilinear form.
His signature formula is:
\begin{theorem}[Wall \cite{wall1969nas}]
\[
\Sign Y=\Sign Y_-+\Sign Y_+ -\Sign(V;B,C,A).
\]
\end{theorem}
Next, we apply his formula to our settings.
We should let $Y_-$ and $Y_+$ denote the manifolds
\begin{eqnarray*}
&Y_-=D^2\times [-1,1]\times [0,1]/\sim\text{ and } Y_+=V_\varphi\times [0,1],
\end{eqnarray*}
respectively.
The rest of the manifolds are 
\begin{eqnarray*}
&\partial Y_-=(\partial D^2\times [-1,1]\times [0,1]/\sim)\cup(D^2\times \{-1,1\}\times [0,1]/\sim),\\
&\partial Y_+=(V_\varphi\times \{1\})\amalg (V_\varphi\times \{0\}),\\
&X_0=\partial D^2\times [-1,1]\times [0,1]/\sim,\quad X_-=D^2\times \{-1,1\}\times [0,1]/\sim,\\
&X_+=(V_\varphi\times \{1\})\amalg (V_\varphi\times\{0\}-\Int j(X_0)),\quad Z=\partial D^2\times \{-1,1\}\times [0,1]/\sim.
\end{eqnarray*}
Consider the case when $T|_{\nu(c)}=id$.
Choose a point $x$ in $\partial D^2$.
Define continuous maps $f_i:S^1\to \partial D^2\times \{-1,1\}\times S^1$ by $f_i(t)=(x,(-1)^i,t)$ for $i=1,2$.
Then, the set consisting of the homology classes
$e_1=[\partial D^2\times \{-1\}]$, $e_2=[\partial D^2\times \{1\}]$, $e_3=[f_1]$, and $e_4=[f_2]$ in $H_1(Z;\mathbb{Q})$ forms a basis.

When $c$ is separating, we have $A=C=\mathbb{Q}e_1\oplus \mathbb{Q}e_2$. Hence, we obtain $W=(B\cap (C+A))/((B\cap C)+(B\cap A))=0$. When $c$ is non-separating, $\Sign (V_\varphi\times[0,1]\cup R)$ is calculated in \cite[Lemma 3.4]{sato2008cfm}.

Consider the case when $T|_{\nu(c)}=r$.
In this case, the curve $c$ is non-separating.
Define a continuous map $f:S^1\to \partial D^2\times \{-1,1\}\times[0,1]/\sim$ by
\[
f(t)=
\begin{cases}
(x,-1,2t)& \text{when }0\le t \le \displaystyle \frac{1}{2},\\
(x,1,2t-1)& \text{when }\displaystyle\frac{1}{2}\le t \le 1.\\
\end{cases}
\]
The set of homology classes consisting of $e_1=[\partial D^2\times \{-1\}]$ and $e_2=[f]$ in $H_1(Z;\mathbb{Q})$ forms a basis.
In this case, $A=B=\mathbb{Q}e_1$.
Hence, we have $W=0$.

\end{proof}

\subsection{The homomorphism $h_{g,c}$}\label{section:homomorphism h_{g,c}}

Let $c$ be a simple closed curve in $\Sigma_g$.
Since the neighborhood $\nu(c)$ of $c$ is diffeomorphic to $\partial D^2\times[-1,1]$, 
we obtain a manifold $L(c)=\Sigma_g\times[0,1]\cup_{\nu(c)} (D^2\times[-1,1])$ by gluing $D^2\times[-1,1]$ along $\nu(c)$.
This is diffeomorphic to a fiber of the projection $(V_\varphi\times[0,1])\cup R\to S^1$ to the last factor.
We denote $\tilde{V}_{\varphi}=(V_{\varphi}\times[0,1])\cup R$ in the following.

Let $\varphi$ and $\psi$ be mapping classes in $\mathcal{M}_g(c)$.
For example, by gluing the $L(c)$-bundles $\tilde{V}_{\varphi}\times [0,1]$ and $\tilde{V}_{\psi}\times [0,1]$ on an annulus,
we obtain a $L(c)$-bundle over $S^2-\amalg_{i=1}^3 \Int D^2$
whose fiberwise boundary is $E_{\varphi,\psi}\amalg -E_{\Phi(\varphi),\Phi(\psi)}$ and
the whole boundary is
\[
(E_{\varphi,\psi}\amalg -E_{\Phi(\varphi),\Phi(\psi)})\cup_{\partial E_{\varphi,\psi}\amalg -\partial E_{\Phi(\varphi),\Phi(\psi)}} (-\tilde{V}_\varphi\amalg -\tilde{V}_\psi\amalg -\tilde{V}_{(\varphi\psi)^{-1}}).
\]
Hence, we have
\[
\Sign E_{\varphi,\psi}-\Sign E_{\Phi(\varphi),\Phi(\psi)}-\Sign \tilde{V}_\varphi-\Sign \tilde{V}_\psi- \Sign \tilde{V}_{(\varphi\psi)^{-1}}=0.
\]
If we rewrite it by Meyer's signature cocycle and the function $s:\mathcal{M}_g(c)\to\mathbb{Z}$, we have
\[
\begin{cases}
-\tau_g(\varphi,\psi)+\Phi^*\tau_{g-1}(\varphi,\psi)-\delta s(\varphi,\psi)=0\in C^2(\mathcal{M}_g(c);\mathbb{Z})&\text{ if }c\text{ is type }\I,\\
-\tau_g(\varphi,\psi)+\Phi^*(\tau_{h}\times\tau_{g-h})(\varphi,\psi)-\delta s(\varphi,\psi)=0\in C^2(\mathcal{M}_g(c^{\ori});\mathbb{Z})&\text{ if }c\text{ is type }\II_h.
\end{cases}
\]

If we restrict the Meyer cocycles to $\mathcal{H}_g$,
we have $\tau_g=\delta \phi_g\in C^2(\mathcal{H}_g;\mathbb{Q})$, and
$\tau_{g-1}=\delta \phi_{g-1}\in C^2(\mathcal{H}_{g-1};\mathbb{Q})$.
Thus, we have proved:
\begin{lemma}\label{lem:homo-h}
When $c$ is type $\I$, define a function $h_{g,c}:\mathcal{H}_g(c)\to \mathbb{Q}$ by
\[
h_{g,c}(\varphi)=s(\varphi)+\phi_g(\varphi)-\Phi^*\phi_{g-1}(\varphi).
\]
When $c$ is type $\II_h$, define $h_{g,c}:\mathcal{H}_g(c^{\ori})\to \mathbb{Q}$ by
\[
h_{g,c}(\varphi)=s(\varphi)+\phi_g(\varphi)-\Phi^*(\phi_{h}\times\phi_{g-h})(\varphi).
\]
Then, both of these maps are homomorphisms.
\end{lemma}

\begin{proof}[Proof of Proposition \ref{prop:values of h}]
First, consider the case when the vanishing cycle $c$ is type $\I$ in Figure \ref{fig:vanishing cycles}.
Since $h_{g,c}$ is a homomorphism, we have $h_{g,c}(\iota_g)=0$.
The mapping classes $t_{c_i}$ for $i=1,2,\cdots,2g-1$ are mutually conjugate in $\mathcal{H}_g(c)$.
Therefore, we have $h_{g,c}(t_{c_1})=\cdots=h_{g,c}(t_{c_{2g-1}})$. 
By the chain relation,
we have $(t_{c_1}\cdots t_{c_{2g-1}})^{2g}=t_{c_{2g+1}}^2$.
Thus, we obtain $h_{g,c}(t_{c_{2g+1}})=g(2g-1)h_{g,c}(t_{c_1})$. 
Hence, it suffices to show that $h_{g,c}(\sigma_{2g+1})=-g/(2g+1)$.

In \cite[Lemma 3.3]{endo2000mss}, Endo showed that $\phi_g(t_{c_{2g+1}})=(g+1)/(2g+1)$.
Since $\Phi(t_{c_{2g+1}})=1\in\mathcal{M}_{g-1}$,
we have $\Phi^*\phi_{g-1}(t_{c_{2g+1}})=0$.
By Lemma \ref{lem:signature round-cob}, we have
\[
s(t_{c_{2g+1}})=\sign m(t_{c_{2g+1}})=\sign([1:-1])=-1.
\]
Thus, we obtain $h_{g,c}(t_{c_{2g+1}})=-g/(2g+1)$.

Next, consider the case when the vanishing cycle $c$ is type $\II$ in Figure \ref{fig:vanishing cycles}.
When $1\le h\le g-1$, this follows from \cite[Lemma 3.3]{endo2000mss} since $s(t_{c_i})=0$.
When $h=0,g$, $h_{g,c}$ is the zero map since $H^1(\mathcal{H}_g(c);\mathbf{Q})=H^1(\mathcal{H}_g;\mathbf{Q})=0$.
\end{proof}

\subsection{Proof of Theorem \ref{thm:main theorem}}\label{section:localization}

We prepare the hyperelliptic mapping class group of the non-connected surface $f^{-1}(r_i)$,
where the monodromy of it along $\partial_0 A_i$ lies.
Identify $f^{-1}(r_i)$ with some standard surface $S_i=\Sigma_{n_i(1)}\amalg\cdots\amalg \Sigma_{n_i(k_i)}$,
where $n_i(1),\cdots, n_i(k_i)$ are non-negative integers.
We may assume that the action on $f^{-1}(r_i)$ induced by $\iota_g$ in Definition \ref{def:hyperelliptic} coincides with $\iota_{n_i(1)}\amalg\cdots\amalg\iota_{n_i(k_i)}$, 
and the vanishing cycle $d_i$ lies in $\Sigma_{n_i(1)}$ ($n_i(1)=g_i$).
Define groups $\mathcal{H}_{S_i}$ and $\mathcal{H}_{S_i}(d_i)$ by
$\mathcal{H}_{S_i}=\mathcal{H}_{n_i(1)}\times\cdots\mathcal{H}_{n_i(k_i)}$
and $\mathcal{H}_{S_i}(d_i)=\mathcal{H}_{n_i(1)}(d_i)\times\mathcal{H}_{n_i(2)}\times\cdots\mathcal{H}_{n_i(k_i)}$, respectively.
By Definition \ref{def:hyperelliptic},
the monodromy $\tilde{\varphi}_1$ along $\partial_0 A_1$ is contained in $\mathcal{H}_{S_1}(d_1)$.
Denote by $\Phi$ the homomorphism $\Phi_n$ if $d_1$ is non-separating, and $\Phi_s$ if $d_1$ is separating.
As stated in Section \ref{section:round cobordisms},
the monodromy $\tilde{\varphi}_2$ of $f^{-1}(r_2)$ along $\partial_0 A_2$ is the image of $\tilde{\varphi}_1\in\mathcal{H}_{S_1}(d_1)$ under $\Phi:\mathcal{H}_{S_1}(d_1)\to\mathcal{H}_{S_2}$.
By Lemma \ref{lem:baykur}, it is contained in $\mathcal{H}_{S_2}(d_2)$.
Define a natural homomorphism $\Phi_{S_i}:\mathcal{H}_{S_i}(d_i)\to\mathcal{H}_{S_{i+1}}$ by $\Phi_{S_i}(x_1,x_2,\cdots,x_{k_i})=(\Phi(x_1),x_2,\cdots,x_{k_i})$, for $i=1,\cdots,m$, where $\Phi$ denotes $\Phi_n$ if $d_i$ is non-separating, and $\Phi_s$ if $d_i$ is separating.
Inductively, the monodromy $\tilde{\varphi}_i$ along $\partial_0 A_i$ is contained in $\mathcal{H}_{S_i}(d_i)$,
and $\tilde{\varphi}_{i+1}=\Phi_{S_i}(\tilde{\varphi}_i)$.

By the Novikov additivity, we have
\begin{align*}
\Sign M=&\sum_{i=1}^m \Sign f^{-1}(A_i)+\Sign f^{-1}(D_l)+\Sign f^{-1}(D_h)\\
=&\sum_{i=1}^m \Sign f^{-1}(A_i)+\Sign f^{-1}(D_h-\coprod_{j=1}^{n}\Int\nu(y_j))+\sum_{j=1}^{n}\Sign f^{-1}(\nu(y_j)).
\end{align*}

Define the Meyer function $\phi_{S_i}:\mathcal{H}_{S_i}\to\mathbb{Q}$ by
$\phi_{S_i}(x_1,\cdots,x_{k_i})=\sum_{j=1}^{k_i}\phi_{S_i}(x_j)\in C^1(\mathcal{H}_{S_i};\mathbb{Q})$.
Let $\psi_j\in \mathcal{H}_g$ denote the monodromy along the loop $a_j$ around the image $y_i\in D_h$ of each Lefschetz singularity for $j=1,\cdots,n$ in Section \ref{section:monodromy}.
By Lemma \ref{lem:meyer} and Lemma \ref{lem:signature round-cob}, we have
\[
\Sign M=\sum_{i=1}^{m}s(\varphi_i)+\left(-\phi_g(\tilde{\varphi}_1^{-1})-\sum_{j=1}^{n}\phi_g(\psi_j)\right)+\sum_{j=1}^{n}\Sign f^{-1}(\nu(y_j)),
\]
where $\varphi_i$ is the monodromy of $f^{-1}(r_i)\cap M_i$ along $\partial_0A_i$.
Since $f^{-1}(D_l)$ is a trivial bundle, we have $\tilde{\varphi}_{m+1}=1\in\mathcal{H}_{S_{m+1}}$.
Since $\Phi_{S_i}(\tilde{\varphi}_i)=\tilde{\varphi}_{i+1}\in\mathcal{H}_{S_{i+1}}(d_{i+1})$, we have
\[
\sum_{i=1}^m(\phi_{S_i}(\tilde{\varphi}_i)-\Phi_{S_i}^*\phi_{S_{i+1}}(\tilde{\varphi}_i))
=\phi_g(\tilde{\varphi}_1).
\]
Since the Meyer function has the property $\phi_g(\varphi^{-1})=-\phi_g(\varphi)$ (see \cite{endo2000mss}) for any $\varphi\in\mathcal{H}_g$,
we obtain
\[
\Sign M=\sum_{i=1}^{m}\left(s(\varphi_i)+\phi_{S_i}(\tilde{\varphi}_i)-\Phi_{S_i}^*\phi_{S_{i+1}}(\tilde{\varphi}_i)\right)+\sum_{j=1}^{n}(-\phi_g(\psi_j)+\Sign f^{-1}(\nu(y_j))).
\]
By the definition of $\Phi_{S_i}$, we have
\begin{align*}
&\phi_{S_i}(x_1,\cdots,x_{k_i})-\Phi_{S_{i+1}}^*\phi_{S_{i+1}}(x_1,\cdots,x_{k_{i+1}})\\
=&
\begin{cases}
\phi_{g_i}(x_1)-\Phi_n^*\phi_{g_i}(x_1), & \text{ if } d_i \text{ is nonseparating},\\
\phi_{g_i}(x_1)-\Phi_s^*(\phi_{h}\times\phi_{g_i-h})(x_1), &\text{ if } d_i \text{ bounds subsurfaces of genus } h \text{ and }g_i-h.
\end{cases}
\end{align*}
Thus, we have
\[
\Sign M=\sum_{i=1}^{m}h_{g_i,d_i}(\varphi_i)+\sum_{j=1}^{n}\sigma_{\loc}(f^{-1}(y_j)).
\]

\section{Examples}\label{section:example}

Let $c_1,\ldots,c_n\subset\Sigma_g$ be simple closed curves described in Figure \ref{fig:scc}. 

\begin{example}\label{ex:exmp1}
As shown in the proof of \cite[Theorem 1.4]{hayano2011nsb}, there exists a simplified BLF $f_{g,n}:M_{g,n}\rightarrow S^2$ which has the following Hurwitz system: 
\[
(t_{c_{2g}}\cdots t_{c_2}t_{c_1}^2t_{c_2}\cdots t_{c_{2g}})^{2n},
\]
and the vanishing cycle of the indefinite fold is $c_{2g+1}$.
\end{example}
By the definition of $f_{g,n}$, it is hyperelliptic. 
We denote by $y_1,\ldots,y_{8gn}\in S^2$ the critical values of $f_{g,n}$. 
By using the formula in Theorem \ref{thm:main theorem}, the signature of $M_{g,n}$ can be calculated as follows: 
{\allowdisplaybreaks
\begin{align*}
\Sign{M_{g,n}} & = \sum_{i=1}^{8gn}\sigma_{\text{loc}}(f_{g,n}^{-1}(y_i)) + h((t_{c_{2g}}\cdots t_{c_2}{t_{c_1}}^2t_{c_2}\cdots t_{c_{2g}})^{2n}) \\
& = 8gn\cdot \frac{-g-1}{2g+1} + h (t_{c_{2g+1}}^{-4n}) \\
& = \frac{-8g^2n-8gn}{2g+1} + (-4n)\cdot \frac{-g}{2g+1} \\
& = -4gn. 
\end{align*}}
It is easy to see that $M_{g,n}$ is simply connected and that the Euler characteristic of $M_{g,n}$ is $8gn-4g+6$. 
As shown in \cite{hayano2011nsb}, $M_{g,n}$ is spin if and only if both of the integers $g$ and $n$ are even. 
Thus, by Freedman's theorem, $M_{g,n}$ is homeomorphic to $\# \dfrac{gn}{4} E(2)\# (\dfrac{5gn}{4}-2g+2) S^2\times S^2$ if both $g$ and $n$ are even 
and $\# (2gn-2g+2)\mathbb{CP}^2 \# (6gn-2g+2) \overline{\mathbb{CP}^2}$ otherwise.

\begin{example}
As shown in the proof of \cite[Theorem 1.4]{hayano2011nsb}, there exists a simplified BLF $\tilde{f}_{g,n}:\tilde{M}_{g,n}\rightarrow S^2$ which has the following Hurwitz system: 
\[
(t_{c_{2g}}\cdots t_{c_2}t_{c_1}^2t_{c_2}\cdots t_{c_{2g}})^{2n}\cdot (t_{c_1}\cdots t_{c_{2g-2}})^{2(2g-1)n},
\]
and a vanishing cycle of the indefinite fold is $c_{2g+1}$.
\end{example}
By the definition of $\tilde{f}_{g,n}$, it is hyperelliptic. 
We denote by $\tilde{y}_1,\ldots,\tilde{y}_{8g^2n-4gn+4n}\in S^2$ the critical values of $\tilde{f}_{g,n}$. 
By using the formula in Theorem \ref{thm:main theorem}, the signature of $\tilde{M}_{g,n}$ can be calculated as follows: 
{\allowdisplaybreaks
\begin{align*}
\Sign{\tilde{M}_{g,n}} & = \sum_{i=1}^{8g^2n-4gn+4n}\sigma_{\text{loc}}(\tilde{f}_{g,n}^{-1}(\tilde{p}_i)) + h((t_{c_{2g}}\cdots t_{c_2}{t_{c_1}}^2t_{c_2}\cdots t_{c_{2g}})^{2n}\cdot (t_{c_1}\cdots t_{c_{2g-2}})^{2(2g-1)n}) \\
& = (8g^2n-4gn+4n)\cdot \frac{-g-1}{2g+1} + 2n \cdot h (t_{c_{2g+1}}^{-2}\cdot \iota_g) +2(2g-1)n\cdot h(t_{c_1}\cdots t_{c_{2g-2}}) \\
& = \frac{-8g^3n+4g^2n-4gn-8g^2n+4gn-4n}{2g+1} -4n \cdot\frac{-g}{2g+1} + 2(2g-1)n (2g-2)\cdot\frac{-1}{4g^2-1} \\
& = -4g^2n.
\end{align*}}

It is easy to see that $\tilde{M}_{g,n}$ is simply connected,
and that the Euler characteristic of $\tilde{M}_{g,n}$ is $8g^2n-4gn+4n-4g+6$. 
As shown in \cite{hayano2011nsb}, $\tilde{M}_{g,n}$ is spin if and only if $g$ is even. 
Thus, we can easily determine the homeomorphism type of $\tilde{M}_{g,n}$ as in Example \ref{ex:exmp1}.  

\noindent
{\bf Acknowledgments. }
The authors would like to express their gratitude to Hisaaki Endo for his continuous support during the course of this work, Nariya Kawazumi for his helpful comments for the draft of this paper. 
The first author is supported by Yoshida Scholarship 'Master21' and he is grateful to Yoshida Scholarship Foundation for their support. The second author is supported by JSPS Research Fellowships for Young Scientists (22-2364).

\end{document}